\documentclass[12pt]{amsart}

\newtheorem{theorem}{Theorem}[section]

\newtheorem{corollary}[theorem]{Corollary}

\theoremstyle{definition}

\theoremstyle{remark}

\numberwithin{equation}{section}

\newcommand{\be}{\begin{equation}}
\newcommand{\ee}{\end{equation}}

\newcommand{\Fp}{\mathbb{F}_p}
\newcommand{\Z}{\mathbb{Z}}
\newcommand{\EFp}{E(\Fp)}
\newcommand{\Q}{\mathbb{Q}}

\begin{document}

\title{The quadratic character of $1+\sqrt{2}$ and an elliptic curve}


\author{Yu Tsumura}
\address{Department of Mathematics, Purdue University
150 North University Street, West Lafayette, Indiana 47907-2067
}

\email{ytsumura@math.purdue.edu}
\thanks{}

\subjclass[2010]{Primary 11G20; Secondary 11E25}

\date{}

\dedicatory{}

\begin{abstract}
When $p\equiv 1 \pmod 8$, we have a criterion of the quadratic character of $1+\sqrt{2}$, which is related to the class number of $\Q(\sqrt{-p})$.
In this paper, we obtain a similar criterion  using an elliptic curve, which contrasts to the proof using algebraic number theory for the old one.

\end{abstract}

\maketitle


\section{Introduction.}

Let $p\equiv 1 \pmod 8$ be a prime number.
Then by the quadratic reciprocity law, $2$ is a square mod $p$, hence there exists $\sqrt{2}\in\Fp$.
Also we know that $p\equiv 1 \pmod 8$ can be expressed as $p=c^2+8d^2$, where $c, d$ are integers.

In1969, Pierre Barrucand and Harvey Cohn provided the criterion of quadratic character of $1+\sqrt{2}$.
Let $\left(\frac{*}{p}\right)$ be the Jacobi symbol mod $p$. Then they proved the following.

\begin{theorem}
Let $p=c^2+8d^2$ be a prime number, where $c, d$ are integers.
Then we have
\[\left(\frac{1+\sqrt{2}}{p}\right)=1 \Longleftrightarrow d\equiv 0 \pmod 2   \Longleftrightarrow   h(-4p) \equiv 0 \pmod 8,\]
where $h(-4p)$ denotes the class number of $\Q(\sqrt{-p})$.
\end{theorem}
\begin{proof}
See \cite{Barrucand-Cohn} and Proposition 5.15 in \cite{Lemmermeyer}.
\end{proof}

This was proved by the method of algebraic number theory.
In this article, we provide  a slightly different criterion with a different taste proof using properties of an elliptic curve defined by $E: y^2=x^3-x$.
This proof is simple and more accessible than the old one for those who are familiar with the basics of elliptic curves.

\section{A new criterion}
Again let $p\equiv 1 \pmod 8$ be a prime number.
Then we can write $p=a^2+b^2$, where $ a, b$ are integers.
Further we can assume that $a$ is odd, $b$ is even and $a+b\equiv1 \pmod 4$.
Then we show the following.

\begin{theorem}
Let notation be as above.
Then we have
\[ \left(\frac{1+\sqrt{2}}{p}\right)=1 \Longleftrightarrow (a-1)^2+b^2\equiv 0 \pmod{32}\]
\end{theorem}

\begin{proof}
Let $\EFp$ be an elliptic curve defined by $y^2=x^3-x$ over a finite field $\Fp$.
We know that $\#\EFp=(a-1)^2+b^2$.
(See Theorem 4.23, page 115 in \cite{Washington}.)

Since $p\equiv1 \pmod 4$, $-1$ is a square mod $p$ by the quadratic reciprocity law.
So let $i\in \Fp$ be a square root of $-1$.
Then the action of $i$ on a point $(x, y) \in \EFp$ defined by $[i]\cdot (x, y)=(-x, iy)$ is easily seen to be a homomorphism.
Actually, $\EFp$ has complex multiplication by $\Z[i]$. 
(See Example 10.2 in \cite{Washington}.)
Let us denote $\eta=1+i$.
We describe the action of $\eta$ explicitly.
Let $\eta\cdot(x,y)=(x_0,y_0)$. 
We have 
\begin{align*}
\eta\cdot(x,y)&=[1+i]\cdot(x,y)=(x,y)+[i]\cdot(x,y)=(x,y)+(-x,iy).
\end{align*}
Here $i$ in the $y$-coordinate is a square root of $-1$ in $\Fp$.
Then by the elliptic curve addition, it is equal to
\begin{equation}\label{eq:eta1}
\left( \left(\frac{(1-i)y}{2x}\right)^2, y_0 \right)
\end{equation}
\begin{equation}\label{eq:eta2}
=\left(\frac{x^2-m}{2ix} ,y_0 \right),
\end{equation}
where $y_0=\left(\frac{(1-i)y}{2x} \right) (x-x_0)-y$.
Note that by the  equation (\ref{eq:eta1}), the $x$-coordinate $x_0$ of  $\eta \cdot(x,y)$ is a square.
Also by the equation (\ref{eq:eta2}), we have $\deg(\eta)=2$, hence it is easy to see that $\ker(\eta)=\{ \infty, (0, 0) \}$, where $\infty$ denotes the point at infinity (the identity).

Let $x(S)$, $S\subset \EFp$ be the set of $x$-coordinate of $(x, y)\in S\setminus \{\infty\}$.
Now, using (\ref{eq:eta2}), we see that $x(\eta^{-1}(\infty))=\{ 0\}$, $x(\eta^{-2}(\infty))=\{\pm1\}$, $x(\eta^{-3}(\infty))=\{\pm i\}$,
$x(\eta^{-4}(\infty))=\{\pm 1 \pm \sqrt{2}\}$.

Now we prove a claim that $(x_0, y_0)\in \EFp$ has a preimage by $\eta$ in $\EFp$ if and only if $x_0$ is a square mod $\Fp$.
Suppose $x_0$ is a square mod $p$, then we solve equation (\ref{eq:eta2}).
Comparing $x$-coordinate we have $x_0=(x^2-m)/(2ix)$.
Solving for $x$, we have $x=x_0i+\sqrt{m-x_0^2}$.
Since $\left( \frac{x_0}{p} \right)=1$, we have 
\[\left(\frac{m-x_0^2}{p}\right)=\left( \frac{-y_0^2x_0^{-1}}{p} \right)=\left( \frac{-1}{p} \right)\left( \frac{y_0^2}{p} \right)\left( \frac{x_0}{p}\right)=1.\]
Hence $\sqrt{m-x_0}\in \Fp$ and therefore, we have $x\in \Fp$.
Now comparing $y$-coordinate, we have $y=y_0 \left\{\left(\frac{1-i}{2x} \right) (x-x_0)-1 \right\}^{-1}\in \Fp$.
Then we have $\eta \cdot (x, y)=(x_0, y_0)$ for $(x, y)\in \EFp $.

The converse follows from the equation (\ref{eq:eta1}).

We finish the consideration of the action of $\eta$ and now we move to prove the theorem.
Suppose $\left(\frac{1+\sqrt{2}}{p}\right)=1$.
Then since
\begin{equation}\label{eq:Jac}
\left(\frac{1+\sqrt{2}}{p}\right)\left(\frac{1-\sqrt{2}}{p}\right)=\left(\frac{-1}{p}\right)=1,
 \end{equation}
  we have $\left(\frac{1-\sqrt{2}}{p}\right)=1$.
Hence all points whose $x$-coordinate is one of $\pm1\pm \sqrt{2}$ have preimages by  $\eta$ in $\EFp$ by the above claim.
Hence $\eta^{-5}(\infty)$ is a subgroup of $\EFp$.
Since $\#\eta^{-5}(\infty)=32$ (note that $\#\ker(\eta)=2$), we have $32|\#\EFp=(a-1)^2+b^2$, hence we have $(a-1)^2+b^2\equiv 0 \pmod{32}$.

Conversely, suppose  $(a-1)^2+b^2\equiv 0 \pmod{32}$.
Since $32|\#\EFp=(a-1)^2+b^2$ and $\EFp$ is in general cyclic or a product of two cyclic groups, 
we have an element $P\in\EFp$ of order $8$.
Then the $x$-coordinate of $\eta(P)$ or $\eta^2(P)$ is one of $\pm 1 \pm \sqrt{2}$, say it is $Q$.
Hence $Q$ has a preimage of $\eta$.
As we saw in the above claim, this means that the $x$-coordinate of $Q$ is a square mod $p$.
So one of $\pm 1 \pm \sqrt{2}$ is a square.
Hence by (\ref{eq:Jac}), all of them are squares.
Especially, we have $\left(\frac{1+\sqrt{2}}{p}\right)=1$.

\end{proof}

Comparing these two theorems, we immediately get the following, which seems to difficult to prove elementarily.

\begin{corollary}

Let $p\equiv 1\pmod 8$ be a prime number.
Let $p=a^2+b^2=c^2+8d^2$, where $a,b,c,d$ are integers and $a$ is odd, $b$ is even and $a+b\equiv 1\pmod4$.
Then 
\[ (a-1)^2+b^2\equiv 0 \pmod{32} \Longleftrightarrow  d\equiv 0 \pmod 2.\]

\end{corollary}

\bibliographystyle{amsplain}
\bibliography{reciprocity}

\end{document}